\documentclass[preprint,12pt]{elsarticle}

\usepackage{amssymb}
 \usepackage{graphicx}
\usepackage{epsfig}

\usepackage{amsthm}

\newcommand{\sysn}{\left\{\begin{array}{rcl}}
\newcommand{\sysk}{\end{array}\right.}

\newtheorem{theorem}{Theorem}[section]

\theoremstyle{example}

\newtheorem{proposition}[theorem]{Proposition}
\theoremstyle{definition}
\newtheorem{definition}[theorem]{Definition}

\newtheorem{corollary}[theorem]{Corollary}

\journal{...}

\begin{document}

\title{Some observations on a clopen version of the Rothberger property}

\author[affil1]{Manoj Bhardwaj}

\address[affil1]{Department of Mathematics, University of Delhi, New Delhi-110007, India}

\ead[affil1]{manojmnj27@gmail.com}

\author[affil2]{Alexander V. Osipov}

\address[affil2]{Krasovskii Institute of Mathematics and Mechanics, \\ Ural Federal
 University, Yekaterinburg, Russia}

\ead[affil2]{OAB@list.ru}

\begin{abstract}

In this paper, we proved that a clopen version
$S_1(\mathcal{C}_\mathcal{O}, \mathcal{C}_\mathcal{O})$ of the
Rothberger property  and Borel strong measure zeroness are independent.
For a zero-dimensional metric space $(X,d)$, $X$ satisfies
$S_1(\mathcal{C}_\mathcal{O}, \mathcal{C}_\mathcal{O})$ if, and
only if, $X$ has Borel strong measure zero with respect to each metric
which has a same topology as $d$ has. In a zero-dimensional space,
the game $G_1(\mathcal{O}, \mathcal{O})$ is equivalent to the game
$G_1(\mathcal{C}_\mathcal{O}, \mathcal{C}_\mathcal{O})$ and the
point-open game is equivalent to the point-clopen game.  Using
reflections, we obtained that the game
$G_1(\mathcal{C}_\mathcal{O}, \mathcal{C}_\mathcal{O})$ and the
point-clopen game are strategically and Markov dual. An example is
given for a space on which the game $G_1(\mathcal{C}_\mathcal{O},
\mathcal{C}_\mathcal{O})$ is undetermined.

\end{abstract}


\maketitle 

\section{Introduction}\label{sec1}

  In 1938, Rothberger \cite{H34b} (see also \cite{H4}) introduced covering
property in topological spaces.  A space $X$ is said to have
\textit{Rothberger property} if for each sequence $\langle
\mathcal{U}_n : n \in \omega \rangle$ of open covers of $X$ there
is a sequence $\langle V_n : n \in \omega \rangle$ such that for
each $n$, $V_n$ is an element of $\mathcal{U}_n$ and each $x \in
X$ belongs to $V_n$ for some $n$\label{2.1}. This property is
stronger than Lindel\"{o}f and preserved under continuous images.

Usually, each selection principle $S_{1}(\mathcal{A},
\mathcal{B})$ can be associated with some topological game
$G_{1}(\mathcal{A},\mathcal{B})$. So the Rothberger property
$S_{1}(\mathcal{O}, \mathcal{O})$  is associated with the
Rothberger game $G_{1}(\mathcal{O}, \mathcal{O})$.

Let $X$ be a topological space. The Rothberger game
$G_{1}(\mathcal{O}, \mathcal{O})$ played on $X$ is a game with two
players Alice and Bob.

{\bf 1st round:} Alice chooses an open cover $\mathcal{U}_1$ of
$X$. Bob chooses a set $U_1\in\mathcal{U}_1$.

{\bf 2st round:} Alice chooses an open cover $\mathcal{U}_2$ of
$X$. Bob chooses a set $U_2\in\mathcal{U}_2$.

{\bf etc.}

If the family $\{U_n: n\in\omega\}$ is a cover of the space $X$
then Bob wins the game $G_1(\mathcal{O}, \mathcal{O})$. Otherwise,
Alice wins.

A topological space is Rothberger if, and only if, Alice has no
winning strategy in the game $G_1(\mathcal{O}, \mathcal{O})$
\cite{H27a}.

In \cite{H34a} Galvin proved that for a first-countable space $X$
Bob has a winning strategy in $G_{1}(\mathcal{O}, \mathcal{O})$
if, and only if, $X$ is countable.

In this paper, we continue to study the mildly Rothberger-type properties, started in papers \cite{BO1,BO3,BO4}, and, we define a new game - the mildly Rothberger game
$G_1(\mathcal{C}_\mathcal{O}, \mathcal{C}_\mathcal{O})$. In a
zero-dimensional space, the Rothberger game is equivalent to the
mildly Rothberger game. Using reflections, we obtained that
$G_1(\mathcal{C}_\mathcal{O}, \mathcal{C}_\mathcal{O})$ and the
point-clopen game are strategically and Markov dual.

\section{Preliminaries}\label{sec2}

Let $(X,\tau)$ or $X$ be a topological space.  If a set is open
and closed in a topological space, then it is called {\it clopen}.
Let $\omega$ be the first infinite cardinal and $\omega_1$ the
first uncountable cardinal.  For the terms and symbols that we do
not define follow \cite{H10}.

Let $\mathcal{A}$ and $\mathcal{B}$ be collections of open covers of a topological space $X$.

The symbol $S_1(\mathcal{A}, \mathcal{B})$ denotes the selection
hypothesis that for each sequence $\langle \mathcal{U}_n : n \in
\omega \rangle$ of elements of $\mathcal{A}$ there exists a
sequence $\langle U_n : n \in \omega \rangle$ such that for each
$n$, $U_n \in \mathcal{U}_n$ and $\{U_n : n \in \omega\} \in
\mathcal{B}$ \cite{H1}.


In this paper $\mathcal{A}$ and $\mathcal{B}$ will be collections of the following open covers of a space $X$:

$\mathcal{O}$ : the collection of all open covers of $X$.

$\mathcal{C}_\mathcal{O}$ : the collection of all clopen covers of $X$.

\medskip

Clearly, $X$ has the Rothberger property if, and only if, $X$
satisfies $S_{1}(\mathcal{O}, \mathcal{O})$.

\medskip

A space $X$ is said to have {\it  mildly Rothberger property} if
it satisfies the selection principles
$S_1(\mathcal{C}_\mathcal{O}, \mathcal{C}_\mathcal{O})$.

It can be noted that $S_1(\mathcal{O}, \mathcal{O}) \Rightarrow
S_1(\mathcal{C}_\mathcal{O}, \mathcal{C}_\mathcal{O})$ and also
every connected space must satisfy $S_1(\mathcal{C}_\mathcal{O},
\mathcal{C}_\mathcal{O})$. Then the set of real numbers with usual
topology satisfies $S_1(\mathcal{C}_\mathcal{O},
\mathcal{C}_\mathcal{O})$ but it does not satisfy
$S_1(\mathcal{O}, \mathcal{O})$.

Let $(X,\tau)$ be a topological space and $\mathcal{T}_X=\tau\setminus \{\emptyset\}$ be a topology without empty set.

$\bullet$ Let $\mathcal{T}_{X,x}=\{U\in \mathcal{T}_X: x\in U\}$
be the local point-base at $x\in X$.

$\bullet$ Let $\mathcal{P}_X=\{\mathcal{T}_{X,x}: x\in X\}$ be the
collection of local point-bases of $X$.

$\bullet$ Let $\mathcal{C}_{\mathcal{T}_{X,x}}=\{U\in
\mathcal{T}_X:$ $U$ is a clopen set in $X$, $x\in U\}$.

$\bullet$ Let $\mathcal{C}_{X}=\{\mathcal{C}_{\mathcal{T}_{X,x}}:
x\in X\}$.

\section{Results on $S_1(\mathcal{C}_\mathcal{O}, \mathcal{C}_\mathcal{O})$}

\subsection{$S_1(\mathcal{C}_\mathcal{O}, \mathcal{C}_\mathcal{O})$ and Borel strong measure zeroness are independent}
Recall that a set of reals $X$ is {\it null} (or has measure zero)
if for each positive $\epsilon$ there exists a cover $\{I_n\}_{n
\in \omega}$ of $X$ such that $\Sigma_n$ diam$(I_n) < \epsilon$.

To restrict the notion of measure zero or null set, in 1919, Borel \cite{H34} defined a notion stronger than measure zeroness. Now this notion is known as strong measure zeroness or strongly null set.

Borel strong measure zero: $Y$ is {\it Borel strong measure zero}
if there is for each sequence $\langle \epsilon_n : n \in \omega
\rangle$ of positive real numbers a sequence $\langle J_n : n \in
\omega \rangle$ of subsets of $Y$ such that each $J_n$ is of
diameter $< \epsilon_n$, and $Y$ is covered by $\{J_n : n \in
\omega\}$.

But Borel was unable to construct a nontrivial (that is, an uncountable) example of a Borel strong measure zero set. He therefore conjectured that there exists no such examples.

In 1928, Sierpinski observed that every Luzin set is Borel strong measure zero, thus the Continuum
Hypothesis implies that Borel’s Conjecture is false.

Sierpinski asked whether the property of being Borel strong measure zero is preserved under taking homeomorphic (or even continuous) images.

In 1941, the answer given by Rothberger is negative under the Continuum Hypothesis. This lead Rothberger to introduce the following topological version of Borel strong measure zero (which is preserved under taking continuous images).

In 1988, Miller and Fremlin \cite{H56} proved that a space $Y$ has
the Rothberger property ($S_1(\mathcal{O},\mathcal{O})$) if, and
only if, it has Borel strong measure zero with respect to each
metric on $Y$ which generates the topology of $Y$.

Recall that a space $X$ is zero-dimensional if it has a base
consisting clopen sets. Now we show that
$S_1(\mathcal{C}_\mathcal{O}, \mathcal{C}_\mathcal{O})$ and Borel strong measure zeroness are independent to each other. Since the set of
real numbers does not have measure zero, it does not have Borel strong measure zero but it satisfies $S_1(\mathcal{C}_\mathcal{O},
\mathcal{C}_\mathcal{O})$. Since every metric space with Borel strong measure zero must be zero-dimensional and separable,
$S_1(\mathcal{C}_\mathcal{O}, \mathcal{C}_\mathcal{O})$ is
equivalent to $S_1(\mathcal{O}, \mathcal{O})$ (see below Theorem
\ref{a1} ). So by Theorem 6(c) \cite{H56}, there is a subset of
reals with Borel strong measure zero but it does not satisfy
$S_1(\mathcal{C}_\mathcal{O}, \mathcal{C}_\mathcal{O})$.

The proof of the following result easily follows from replacing
the open sets with sets of a clopen base of the topological space.

\begin{theorem} \label{a1}
For a zero-dimensional space $X$, $S_1(\mathcal{C}_\mathcal{O}, \mathcal{C}_\mathcal{O})$ is equivalent to $S_1(\mathcal{O}, \mathcal{O})$.
\end{theorem}

From Theorem 1 in \cite{H56}, we obtain the following corollary.

\begin{corollary} {\it
For a zero-dimensional metric space $(X,d)$ the following
statements are equivalent :
\begin{enumerate}
\item $X$ satisfies $S_1(\mathcal{O}, \mathcal{O})$; \item $X$
satisfies $S_1(\mathcal{C}_\mathcal{O}, \mathcal{C}_\mathcal{O})$;
\item $X$ has Borel strong measure zero with respect to every metric
which generates the original topology; \item every continuous
image of $X$ in Baire space $\omega^\omega$ with usual metric has Borel strong measure zero.
\end{enumerate}}
\end{corollary}

\subsection{Dual selection games}

 The {\it selection game} $G_1(\mathcal{A},\mathcal{B})$ is an
 $\omega$-length game for two players, Alice and Bob. During round $n$,
 Alice choose $A_n\in \mathcal{A}$, followed by Bob choosing
 $B_n\in A_n$. Player Bob wins in the case that $\{B_n:
 n<\omega\}\in \mathcal{B}$, and Player Alice wins otherwise.

 We consider the following strategies:

$\bullet$ A {\it strategy for player Alice} in
$G_1(\mathcal{A},\mathcal{B})$ is a function $\sigma: (\bigcup
\mathcal{A})^{<\omega}\rightarrow \mathcal{A}$. A strategy
$\sigma$ for Alice is called {\it winning} if whenever $x_n\in
\sigma\langle x_i : i<n\rangle$ for all $n<\omega$, $\{x_n: n\in
\omega\}\not\in \mathcal{B}$. If player Alice has a winning
strategy, we write $Alice \uparrow G_1(\mathcal{A},\mathcal{B})$.

$\bullet$ A strategy for player Bob in
$G_1(\mathcal{A},\mathcal{B})$ is a function $\tau:
\mathcal{A}^{<\omega}\rightarrow \bigcup \mathcal{A}$. A strategy
$\tau$ for Bob is {\it winning} if $A_n\in \mathcal{A}$ for all
$n<\omega$, $\{\tau(A_0,...,A_n): n<\omega\}\in \mathcal{B}$.

$\bullet$ A {\it predetermined strategy} for Alice is a strategy
which only considers the current turn number. Formally it is a
function $\sigma: \omega \rightarrow \mathcal{A}$. If Alice has a
winning predetermined strategy, we write $Alice {\uparrow\atop
pre} G_1(\mathcal{A},\mathcal{B})$.

$\bullet$ A {\it Markov strategy} for Bob is a strategy which only
considers the most recent move of player Alice and the current
turn number. Formally it is a function $\tau: \mathcal{A}\times
\omega \rightarrow \bigcup \mathcal{A}$. If Bob has a winning
Markov strategy, we write $Bob {\uparrow\atop mark}
G_1(\mathcal{A},\mathcal{B})$.
\medskip

Note that, $Bob {\uparrow\atop mark} G_1(\mathcal{A},\mathcal{B})$
$\Rightarrow$ $Bob \uparrow G_1(\mathcal{A},\mathcal{B})$
$\Rightarrow$ $Alice \not\uparrow G_1(\mathcal{A},\mathcal{B})$
$\Rightarrow$ $Alice {\not\uparrow\atop pre}
G_1(\mathcal{A},\mathcal{B})$.

It's worth noting that $Alice {\not\uparrow\atop pre}
G_1(\mathcal{A},\mathcal{B})$ is equivalent to the selection
principle $S_1(\mathcal{A},\mathcal{B})$.

\medskip

 Two games $G_1$ and $G_2$ are said to be
{\bf strategically dual} provided that the following two hold:

$\bullet$ $Alice \uparrow G_1$ iff $Bob \uparrow G_2$

$\bullet$ $Alice \uparrow G_2$ iff $Bob \uparrow G_1$.

\medskip

Two games $G_1$ and $G_2$ are said to be {\bf Markov dual}
provided that the following two hold:

$\bullet$ $Alice {\uparrow\atop pre} G_1$ iff $Bob {\uparrow\atop
mark} G_2$

$\bullet$ $Alice {\uparrow\atop pre} G_2$ iff $Bob {\uparrow\atop
mark} G_1$.

\medskip

Two games $G_1$ and $G_2$ are said to be {\bf dual} provided that
they are both strategically dual and Markov dual.

\medskip

 For a set $X$, let $\mathcal{C}(X)=\{f\in (\bigcup X)^X: x\in
 X\Rightarrow f(x)\in x\}$ be the collection of all choice
 functions on $X$.

Write $X\preceq Y$ if $X$ is coinitial in $Y$ with respect to
$\subseteq$; that is, $X\subseteq Y$, and for all $y\in Y$, there
exists $x\in X$ such that $x\subseteq y$.

In the context of selection games, $\mathcal{A}'$ is a {\it
selection basis} for $\mathcal{A}$ when $\mathcal{A}'\preceq
\mathcal{A}$ \cite{clontz}.

\begin{definition}\cite{clontz}
The set $\mathcal{R}$ is said to be a {\bf reflection} of the set
$\mathcal{A}$ if $\{range(f): f\in \mathcal{C}(\mathcal{R})\}$ is
a selection basis for $\mathcal{A}$.

\end{definition}

 Let $G_1(\mathcal{A}, \neg \mathcal{B}):=G_1(\mathcal{A},
 \mathcal{P}(\bigcup \mathcal{A})\setminus \mathcal{B})$.

\begin{theorem}\label{Thcl}$($\cite{clontz}, Corollary 26$)$
If $\mathcal{R}$ is a reflection of $\mathcal{A}$, then
$G_1(\mathcal{A},\mathcal{B})$ and $G_1(\mathcal{R}, \neg
\mathcal{B})$ are dual.
\end{theorem}

The point-open game $PO(X)$ is a game where Alice chooses points
of $X$, Bob chooses an open neighborhood of each chosen point, and
Alice wins if Bob's choices are a cover.

\begin{theorem} \cite{H34a}
The game $G_1(\mathcal{O}, \mathcal{O})$ is strategically dual to
the point-open game on each topological space.
\end{theorem}

\begin{theorem} \cite{clHl}
The game $G_1(\mathcal{O}, \mathcal{O})$ is Markov dual to the
point-open game on each topological space.
\end{theorem}

\medskip

\begin{corollary}{\it The game $G_1(\mathcal{O}, \mathcal{O})$ is dual to the
point-open game on each topological space}.
\end{corollary}

 Recall that two games $G$ and $G^{'}$ are {\it equivalent} (isomorphic) if
\begin{enumerate}
\item $Alice \uparrow G$ iff $Alice \uparrow G^{'}$

\item $Bob \uparrow G$ iff $Bob \uparrow G^{'}$.
\end{enumerate}

\medskip

Since $\mathcal{P}_X$ is a reflection of $\mathcal{O}$
(\cite{clontz}, Proposition 28), the Rothberger game
$G_1(\mathcal{O},\mathcal{O})$ and $G_1(\mathcal{P}_X, \neg
\mathcal{O})$ are dual (\cite{clontz}, Corollary 29). It is well
known that the game $G_1(\mathcal{P}_X, \neg \mathcal{O})$ is
equivalent to the point-open game.

\subsection{The point-clopen and quasi-component-clopen games}

The {\it point-clopen game} $PC(X)$ on a space $X$ is played
according to the following rules:

In each inning $n \in \omega$, Alice picks a point $x_n \in X$, and then Bob chooses a clopen set $U_n \subseteq X$ with $x_n \in U_n$. At the end of the play
\begin{center}
$x_0, U_0, x_1, U_1, x_2, U_2, . . . , x_n, U_n, . . . $,
\end{center}
the winner is Alice if $X=\bigcup_{n \in \omega} U_n$, and Bob otherwise.

We denote the collection of all non-empty clopen subsets of a space $X$ by $\tau_{c}$ and the collection of all finite subsets of $\tau_{c}$ by $\tau^{< \omega}_{c}$.

A strategy for Alice in the point-clopen game on a space $X$ is a function $\varphi : \tau^{< \omega}_{c} \rightarrow X$.

A strategy for Bob in the point-clopen game on a space $X$ is a function $\psi : X^{< \omega} \rightarrow \tau_{c}$ such that, for all $\langle x_0, x_1,..., x_n \rangle \in X^{< \omega} \setminus \{\langle \rangle\}$, we have $x_n \in \psi(\langle x_0,..., x_n \rangle$) = $U_n$.

A strategy $\varphi : \tau^{< \omega}_{c} \rightarrow X$ for Alice
in the point-clopen game on a space $X$ is a winning strategy for Alice
if, for every sequence $\langle U_n : n \in \omega \rangle$ of
clopen subsets of a space $X$ such that $\forall n \in \omega,
(x_n = \varphi(\langle U_0, U_1,..., U_{n-1} \rangle) \in U_n)$,
we have $X=\bigcup_{n \in \omega} U_n$. If Alice has a
winning strategy in the point-clopen game on a space $X$, we write $Alice
{\uparrow} PC(X)$.

A strategy $\psi : X^{< \omega} \rightarrow \tau_{c}$ for Bob in the point-clopen game on a space $X$ is a
 winning strategy for Bob if, for every sequence $\langle x_n : n \in \omega \rangle$ of points of a space $X$,
 we have $X=\bigcup_{n \in \omega}\{U_n: U_n=\psi(\langle x_0, x_1,..., x_n \rangle)\}$. If Bob has a
winning strategy in the point-clopen game on a space $X$, we write $Bob
{\uparrow} PC(X)$.

The game $G_1(\mathcal{C}_\mathcal{O}, \mathcal{C}_\mathcal{O})$ is a game for two players, Alice and Bob, with an inning per each natural number $n$. In each inning, Alice picks a clopen cover of the space and Bob selects one member from this cover. Bob wins if the sets he selected throughout the game cover the space. If this is not the case, Alice wins.

\medskip
 The intersection of all clopen sets containing a
component is called a {\it quasi-component} of the space
\cite{enc}.

The {\it quasi-component-clopen game} $QC(X)$ on a space $X$ is
played according to the following rules :

In each inning $n \in \omega$, Alice picks a quasi-component $A_n$
of $X$, and then Bob chooses a clopen set $U_n \subseteq X$ with
$A_n \subseteq U_n$. At the end of the play
\begin{center}
$A_0, U_0, A_1, U_1, A_2, U_2, . . . , A_n, U_n, . . . $,
\end{center}
the winner is Alice if $X=\bigcup_{n \in \omega} U_n$,
and Bob otherwise.

We denote the collection of all quasi-components of a space $X$ by
$Q_X$ and the collection of all finite subsets of $Q_X$ by $Q_X^{<
\omega}$.

A strategy for Alice in the quasi-component-clopen game on a space
$X$ is a function $\varphi : \tau^{< \omega}_{c} \rightarrow Q_X$.

A strategy for Bob in the quasi-component-clopen game on a space
$X$ is a function $\psi : Q_X^{< \omega} \rightarrow \tau_{c}$
such that, for all $\langle A_0, A_1,..., A_n \rangle \in Q_X^{<
\omega} \setminus \{\langle \rangle\}$, we have $A_n \subseteq
\psi(\langle A_0,..., A_n \rangle$) = $U_n$.

A strategy $\varphi : \tau^{< \omega}_{c} \rightarrow Q_X$ for
Alice in the quasi-component-clopen game on a space $X$ is a winning
strategy for Alice if, for every sequence $\langle U_n : n \in
\omega \rangle$ of clopen subsets of a space $X$ such that
$\forall n \in \omega, (A_n = \varphi(\langle U_0, U_1,...,
U_{n-1} \rangle) \subseteq U_n)$, we have $X=\bigcup_{n
\in \omega} U_n$. If Alice has a winning strategy in the
quasi-component-clopen game on a space $X$, we write $Alice {\uparrow}
QC(X)$.

A strategy $\psi : Q_X^{< \omega} \rightarrow \tau_{c}$ for Bob in
the quasi-component-clopen game on a space $X$ is a winning strategy for
Bob if, for every sequence $\langle A_n : n \in \omega \rangle$ of
quasi-components of a space $X$, we have $X=\bigcup_{n
\in \omega}\{U_n: U_n=\psi(\langle A_0, A_1,..., A_n \rangle)\}$. If
Bob has a winning strategy in the quasi-component-clopen game on a space
$X$, we write $Bob {\uparrow} QC(X)$.

\begin{proposition} \label{b1} The point-clopen game is equivalent to the quasi-component-clopen game.
\end{proposition}

\begin{proof} Let $\varphi : \tau^{< \omega}_{c} \rightarrow X$ be
a winning strategy for Alice in the point-clopen game on a space
$X$. Then the function $\psi : \tau^{< \omega}_{c} \rightarrow
Q_X$ such that $\psi(\langle U_0, U_1,..., U_{n-1}
\rangle)=Q[\varphi(\langle U_0, U_1,..., U_{n-1} \rangle)]$
($Q[x]$ is the quasi-component of $x$) for every sequence $\langle
U_n : n \in \omega \rangle$ of clopen subsets of a space $X$ and
$n \in \omega$,  is a winning strategy for Alice in the
quasi-component-clopen game. This follows from the fact that
$x_n=\varphi(\langle U_0, U_1,..., U_{n-1} \rangle)\in
Q[x_n]\subseteq U_n$.

Let $\varphi : \tau^{< \omega}_{c} \rightarrow Q_X$ be a winning
strategy for Alice in the quasi-component-clopen game on a space
$X$. Then the function $\psi : \tau^{< \omega}_{c} \rightarrow X$
such that $\psi(\langle U_0, U_1,..., U_{n-1} \rangle)\in
\varphi(\langle U_0, U_1,..., U_{n-1} \rangle)$ for every sequence
$\langle U_n : n \in \omega \rangle$ of clopen subsets of a space
$X$ and $n \in \omega$,  is a winning strategy for Alice in the
point-clopen game.

Let $\psi : X^{< \omega} \rightarrow \tau_{c}$ be a winning
strategy for Bob in the point-clopen game on $X$. Then the
function $\rho : Q_X^{< \omega} \rightarrow \tau_{c}$ such that
$\rho(\langle A_0, A_1,..., A_n \rangle)=\psi(\langle x_0,
x_1,..., x_n \rangle)$ for every sequence $\langle A_n : n \in
\omega \rangle$ of quasi-components of a space $X$ and some
$x_0,...,x_n$ that $A_i=Q[x_i]$ for each $i=0,...,n$, is a winning
strategy for Bob in the quasi-component-clopen game.

Let $\psi : Q_X^{< \omega} \rightarrow \tau_{c}$ be a winning
strategy for Bob in the quasi-component-clopen game on $X$. Then
the function $\rho : X^{< \omega} \rightarrow \tau_{c}$ such that
$\rho(\langle x_0, x_1,..., x_n \rangle)=\psi(\langle A_0,
A_1,..., A_n \rangle)$ for every sequence $\langle x_n : n \in
\omega \rangle$ of points of a space $X$ where $A_i=Q[x_i]$ for
each $i=0,...,n$, is a winning strategy for Bob in the
point-clopen-clopen game.
\end{proof}

\begin{proposition} $\mathcal{C}_{X}$ is a reflection of
$\mathcal{C}_\mathcal{O}$.
\end{proposition}

\begin{proof} For every clopen cover $\mathcal{U}$, the
corresponding choice function $f\in \mathcal{C}(\mathcal{C}_{X})$
is simply the witness that $x\in
f(\mathcal{C}_{\mathcal{T}_{X,x}})\in\mathcal{U}$.

\end{proof}

By Theorem \ref{Thcl}, we get the following result.

\begin{corollary}\label{cor1} $G_1(\mathcal{C}_\mathcal{O},\mathcal{C}_\mathcal{O})$
and $G_1(\mathcal{C}_X, \neg \mathcal{C}_\mathcal{O})$ are dual.
\end{corollary}

Note that $PC(X)$ and $G_1(\mathcal{C}_X, \neg \mathcal{C}_\mathcal{O})$ are the same game.

By Proposition \ref{b1}, $PC(X)$ and $QC(X)$ are equivalent, hence, we get the following result.

\begin{proposition}\label{pr1} The game $G_1(\mathcal{C}_X, \neg \mathcal{C}_\mathcal{O})$ is equivalent to the quasi-component-clopen game.
\end{proposition}

\begin{corollary} \label{b3}
{\it If a space $X$ is a union of countable number of
quasi-components, then Bob $\uparrow G_1(\mathcal{C}_\mathcal{O},
\mathcal{C}_\mathcal{O})$}.
\end{corollary}

The following chain of implications always holds:

\begin{center}

$X$ is a union of countable number of quasi-components \\
$\Downarrow$ \\ $Bob \uparrow G_1(\mathcal{C}_\mathcal{O},
\mathcal{C}_\mathcal{O})$ \\ $\Downarrow$ \\ $Alice \not\uparrow
G_1(\mathcal{C}_\mathcal{O},
\mathcal{C}_\mathcal{O})$ \\ $\Updownarrow$ \\
$X$ has mildly Rothberger property.

\end{center}

The proof of the following result easily follows from replacing
the open sets with sets of a clopen base of the topological space.

\begin{theorem}
For a zero-dimensional space, the following statements hold:
\begin{enumerate}
\item
The game $G_1(\mathcal{C}_\mathcal{O}, \mathcal{C}_\mathcal{O})$ is equivalent to the game $G_1(\mathcal{O}, \mathcal{O})$.
\item
The point-clopen game is equivalent to the point-open game.
\end{enumerate}
\end{theorem}

From \cite{H27a} and \cite{H42a}, we have the following result.

\begin{theorem}\label{th314}
For a space $X$, the following statements hold :
\begin{enumerate}
\item \cite{H27a} $X$ satisfies $S_1(\mathcal{O}, \mathcal{O})$
iff $Alice \not\uparrow G_1(\mathcal{O}, \mathcal{O})$.

 \item \cite{H42a} $X$ satisfies
$S_1(\mathcal{C}_\mathcal{O}, \mathcal{C}_\mathcal{O})$ iff $Alice
\not\uparrow G_1(\mathcal{C}_\mathcal{O},
\mathcal{C}_\mathcal{O})$.
\end{enumerate}
\end{theorem}

\begin{corollary}
For a space $X$, the following statements are equivalent :
\begin{enumerate}
\item $X$ satisfies $S_1(\mathcal{C}_\mathcal{O},
\mathcal{C}_\mathcal{O})$; 

\item $Alice {\not\uparrow\atop pre} G_1(\mathcal{C}_\mathcal{O},
\mathcal{C}_\mathcal{O})$;

\item $Alice \not\uparrow G_1(\mathcal{C}_\mathcal{O},
\mathcal{C}_\mathcal{O})$; \item $Bob \not\uparrow
G_1(\mathcal{C}_X, \neg \mathcal{C}_\mathcal{O})$;

\item $Bob {\not\uparrow\atop mark} G_1(\mathcal{C}_X, \neg
\mathcal{C}_\mathcal{O})$;

\item $Bob \not\uparrow PC(X)$; \item $Bob \not\uparrow QC(X)$;

\item $Bob {\not\uparrow\atop mark} PC(X)$;

\item $Bob {\not\uparrow\atop mark} QC(X)$.

\end{enumerate}
\end{corollary}

\begin{corollary}
For a zero-dimensional space $X$, the following statements are equivalent :
\begin{enumerate}
\item $X$ satisfies $S_1(\mathcal{O}, \mathcal{O})$; \item $X$
satisfies $S_1(\mathcal{C}_\mathcal{O}, \mathcal{C}_\mathcal{O})$;

\item $Alice {\not\uparrow\atop pre} G_1(\mathcal{C}_\mathcal{O},
\mathcal{C}_\mathcal{O})$;

\item $Alice {\not\uparrow\atop pre} G_1(\mathcal{O},
\mathcal{O})$;

\item $Alice \not\uparrow G_1(\mathcal{O}, \mathcal{O})$; \item
$Alice \not\uparrow G_1(\mathcal{C}_\mathcal{O},
\mathcal{C}_\mathcal{O})$;

\item $Bob \not\uparrow G_1(\mathcal{P}_X, \neg \mathcal{O})$;

\item $Bob \not\uparrow G_1(\mathcal{C}_X, \neg
\mathcal{C}_\mathcal{O})$; \item $Bob \not\uparrow PO(X)$; \item
$Bob \not\uparrow PC(X)$; \item $Bob \not\uparrow QC(X)$;

\item $Bob {\not\uparrow\atop mark} PO(X)$; \item $Bob
{\not\uparrow\atop mark} PC(X)$; \item $Bob {\not\uparrow\atop
mark} QC(X)$.
\end{enumerate}
\end{corollary}

In \cite{H34a}, Galvin and
Telg\'{a}rsky (Theorem 6.3 in \cite{tel}) proves: If $X$ is a Lindel\"{o}f space in which each
element is $G_{\delta}$, then Bob has a winning strategy in
$G_1(\mathcal{O}, \mathcal{O})$ if, and only if, $X$ is countable.


\begin{theorem}
Let $X$ be a space in which each quasi-component is an
intersection of countably many clopen sets, then $Bob \uparrow
G_1(\mathcal{C}_\mathcal{O}, \mathcal{C}_\mathcal{O})$ if, and
only if, $X$ is a union of countably many quasi-components.
\end{theorem}

\begin{proof}
Let Bob have a winning strategy in the game
$G_1(\mathcal{C}_\mathcal{O}, \mathcal{C}_\mathcal{O})$ on $X$. Since the game $G_1(\mathcal{C}_\mathcal{O},
\mathcal{C}_\mathcal{O})$ and the point-clopen game are dual  and,
by Proposition \ref{b1}, the point-clopen game and the
quasi-component-clopen game are equivalent.

Let Alice have a winning strategy
in the quasi-component-clopen game. Let $\varphi$ be a winning
strategy of Alice in the quasi-component-clopen game on $X$. For
every quasi-component $Q$, there is a sequence $\langle
V_k : k \in \omega \rangle$ of clopen sets such that $Q=
\cap_{k \in \omega} V_k$.

 So we restrict the move of Bob from
$\{V_k: k \in \omega\}$ for $Q$ played by Alice.

Let Alice start the play of the point-clopen game by quasi-component $\varphi(\langle \rangle) = Q_{\langle \rangle}$. Then Bob replies with a clopen set of the form $V_{k_0, \langle \rangle}$ for some $k_0 \in \omega$.

Alice's next move in the play is a quasi-component $\varphi(\langle V_{k_0, \langle \rangle} \rangle) = Q_{\langle k_0 \rangle}$. Then Bob replies with a clopen set of the form $V_{k_1, \langle k_0 \rangle}$ for some $k_1 \in \omega$.

Now Alice's next move in the play is a quasi-component $\varphi(\langle V_{k_0, \langle \rangle}, V_{k_1, \langle k_0 \rangle} \rangle) = Q_{\langle k_0, k_1 \rangle}$. Then Bob replies with a clopen set of the form $V_{k_2, \langle k_0, k_1 \rangle}$ for some $k_2 \in \omega$ and so on.

Similarly we are defining $\langle Q_s : s \in \omega^{< \omega} \rangle$ by setting $Q_{\langle \rangle} = \varphi(\langle \rangle)$ and for each $s \in \omega^{< \omega}$ and for each $k \in \omega$, defining
\begin{center}
$Q_{s \frown \langle k \rangle} = \varphi(\langle V_{s(0), s \upharpoonright 0}, V_{s(1), s \upharpoonright 1},...,V_{s(m-1), s \upharpoonright (m-1)}, V_{k,s} \rangle)$,
\end{center}
where $m = dom(s)$.
From this we construct a countable collection $\{ Q_s : s \in \omega^{< \omega} \}$.

Now to show that $\bigcup \{ Q_s : s \in \omega^{< \omega} \} = X$.
If possible suppose that $\bigcup \{ Q_s : s \in \omega^{< \omega} \} \neq X$, then there is $y \in X \setminus \{ Q_s : s \in \omega^{< \omega} \}$. Then $y \notin Q_s$ for any $s \in \omega^{< \omega}$. For each $Q_n \in \{ Q_s : s \in \omega^{< \omega} \}$, there is some $k_n$ such that $y \notin V_{k_n, n}$. Then Alice loses the following play of the quasi-component-clopen game
\begin{center}
$\langle Q_0, V_{k_0,0}, Q_1, V_{k_1,1},..., Q_n, V_{k_n, n},...\rangle$
\end{center}
in which Alice uses the strategy $\varphi$ since $y \notin \bigcup_{n \in \omega} V_{k_n, n}$, a contradiction.

Converse follows from Corollary \ref{b3}.
\end{proof}

\subsection{Determinacy and $G_1(\mathcal{C}_\mathcal{O}, \mathcal{C}_\mathcal{O})$ game}
A game $G$ played between two players Alice and Bob is determined if either Alice has a winning strategy in game
$G$ or Bob has a winning strategy in game $G$. Otherwise $G$ is undetermined.

It can be observed that the game $G_1(\mathcal{C}_\mathcal{O},
\mathcal{C}_\mathcal{O})$ is determined for every countable space.
But in a mildly Rothberger space in which each quasi-component is
an intersection of countably many clopen sets with uncountable
many quasi-components, none of the players Alice and Bob have a
winning strategy. So $G_1(\mathcal{C}_\mathcal{O},
\mathcal{C}_\mathcal{O})$ is undetermined for a mildly Rothberger
space in which each quasi-component is an intersection of
countably many clopen sets with uncountable many quasi-components.
Thus every uncountable zero-dimensional mildly Rothberger metric
space is undetermined.

Recall that an uncountable set $L$ of reals is a {\it Luzin set}
if for each meager set $M$, $L \cap M$ is countable. The Continuum
Hypothesis implies the existence of a Luzin set. A Luzin set is an
example of a space for which the game
$G_1(\mathcal{C}_\mathcal{O}, \mathcal{C}_\mathcal{O})$ is
undetermined.

\medskip

{\bf Acknowledgements.} The authors would like to thank the
referees for careful reading and valuable comments. The work was performed as part of research conducted in the Ural Mathematical Center with the financial support
of the Ministry of Science and Higher Education of the Russian
Federation (Agreement number 075-02-2023-913).


\begin{thebibliography}{16}




\bibitem{H34} E. Borel, \textit{Sur la classification des ensembles de mesure nulle}, Bull. Soc. Math. France, \textbf{47} (1919), 97--125.


\bibitem{H42a} M. Bhardwaj and B. K. Tyagi, \textit{The Mildly Menger and Mildly Rothberger games}, communicated.

\bibitem{BO1} M. Bhardwaj, A.V. Osipov, \textit{Mildly version of Hurewicz basis covering property and Hurewicz measure zero spaces}, Bulletin of the Belgian Mathematical Society - Simon Stevin, 29(1), (2022) 121--131.

\bibitem{BO3} M. Bhardwaj, A.V. Osipov, \textit{Some Observations on the Mildly Menger Property and Topological Games}, Filomat, 36(15), (2022) 5289--5296.

 \bibitem{BO4} M. Bhardwaj, A.V. Osipov, \textit{Star versions of the Hurewicz basis covering property and strong measure zero spaces}. Turkish Journal of Mathematics, 44(3), (2020) 1042--1053.


\bibitem{clHl} S. Clontz and J. Holshouser, \textit{Limited
information strategies and discrete selectivity}, Topology Appl.,
\textbf{265} (2019), 106815.

\bibitem{clontz} S. Clontz, \textit{Dual selection games},
Topology Appl., \textbf{272} (2020), 107056.

\bibitem{H10} R. Engelking, \textit{General Topology, Revised and completed edition}, Heldermann Verlag Berlin (1989).



\bibitem{H34a} F. Galvin, \textit{Indeterminacy of point-open games}, Bull. Acad. Pol. Sci., \textbf{26} (1978), 445--449.

\bibitem{enc}
K.P.Hart, Jun-iti Nagata, J.E.Vaughan, \textit{Encyclopedia of
General Topology}, Elsevier Science, 2003, 536 p. (Jerzy
Mioduszewski, \textit{d-21 Connectedness}).



\bibitem{H4} W. Hurewicz, \textit{$\ddot{U}$ber eine verallgemeinerung des Borelschen Theorems}, Math. Z. \textbf{24} (1925), 401--421.





\bibitem{H56} A.W. Miller and D.H. Fremlin, \textit{Some properties of Hurewicz, Menger and Rothberger}, Fund. Math., \textbf{129} (1988), 17--33.

\bibitem{H27a} J. Pawlikowski, \textit{Undetermined sets of point-open games}, Fund. Math., \textbf{144} (1994) 279-285.


\bibitem{H34b} F. Rothberger, \textit{Eine Versch$\ddot{o}$rfung der Eigenschaft C}, Fund. Math., \textbf{30} (1938) 50--55.


\bibitem{H1} M. Scheepers, \textit{Combinatorics of open covers (I) : Ramsey theory},  Topology Appl.,\textbf{69} (1996), 31--62.

\bibitem{tel} R. Telg\'{a}rsky, \textit{Spaces defined by
topological games}, Fund. Math. {\bf 88}:3 (1975), 193--223.

\end{thebibliography}
\end{document}